\documentclass[11pt,twoside]{article}
\usepackage{latexsym,amsmath,amssymb,times}

\newif\ifdviwin

\usepackage{titlesec}
	
\titleformat{\section}{\large\bfseries\center}{\thesection}{1em}{\vspace{-.3cm}}
\titleformat{\subsection}[runin]{\bfseries}{\thesubsection}{1em}{}

\usepackage[english]{babel}
\usepackage{indentfirst}
\usepackage[mathscr]{eucal}
\usepackage{amssymb,amsmath,amsfonts}
\usepackage{fancybox,fancyhdr}
\usepackage{graphicx}
\usepackage[utf8]{inputenc}
\usepackage{mathrsfs}
\usepackage{float}
\usepackage{color}
\usepackage{upgreek}
\usepackage[pdftex]{hyperref}
\hypersetup{colorlinks=true,linkcolor=blue,citecolor=blue}

\usepackage{enumitem}
\setlist[enumerate]{
  topsep=6pt,      % espacio antes y después de la lista
  partopsep=0pt,   % espacio extra si la lista empieza un párrafo
  itemsep=0pt,     % espacio entre ítems
  parsep=4pt,      % espacio entre párrafos *dentro* de un mismo ítem
  leftmargin=*,    % opcional: ajusta el margen
}

\newif\ifdviwin

\dviwintrue

\def\r{\mathbb R}
\def\h{\mathbb H}
\def\s{\mathbb{S}}

\def\ekt{\mathbb{E}(\kappa,\tau)}
\def\rkt{\mathcal{R}(\kappa,\tau)}
\def\ricn{\mathrm{Ric}(N)}
\def\onabla{\overline{\nabla}}
\def\Q{\mathcal{Q}}
\def\D{\mathbb{D}}
\def\t{\textbf{t}}
\def\im{\mathrm{Im}}
\def\n{\textbf{n}}
\def\div{\mathrm{div}}
\def\z{\zeta}
\def\m{\mathbb M}

 \newtheorem{defi}{Definition}[section]
 \newtheorem{teo}{Theorem}
  
 \newtheorem{pro}[defi]{Proposition}
 \newtheorem{cor}[defi]{Corollary}
 \newtheorem{lem}[defi]{Lemma}
 
 \newtheorem{remark}[defi]{Remark}

 \newenvironment{proof}{\rm \trivlist \item[\hskip \labelsep{\it
      Proof}.]}{\nopagebreak \hfill $\Box$ \endtrivlist}

\numberwithin{equation}{section}

\begin{document}
\thispagestyle{empty}
%\mbox{}\vspace{0.4cm}

\begin{center}
%\rule{16cm}{1.5pt}\vspace{0.5cm}

\renewcommand{\thefootnote}{\,}
{\Large \bf Stable constant mean curvature disks with circular boundary in homogeneous three-manifolds
\footnote{\hspace{-.75cm}

}}\\
\vspace{0.5cm} {Antonio Bueno}\\
%\rule{16cm}{1.5pt}
\end{center}
\vspace{.5cm}

\begin{abstract}
Let $H>0$ and $\kappa,\tau\in\r$ satisfy $4H^2+\kappa>0$ and $\kappa\leq 4\tau^2$. We prove that an immersed stable disk of constant mean curvature $H$ and circular boundary in the space $\ekt$ is a spherical cap of the canonical rotational $H$-sphere. 
\end{abstract}

\section{Introduction}

The characterization of the round sphere as the only closed surface of positive constant mean curvature (CMC) under further geometric assumptions is a classical topic of research in Differential Geometry in the large. Among all the literature regarding this question, we highlight three results that have reshaped and motivated different lines of research. We recall them by chronological appearance. 

In 1955, Hopf \cite{hop} proved that a CMC immersion of a topological sphere in $\r^3$ must be totally umbilical. The proof is based on the \emph{Hopf differential}, a quadratic differential that vanishes at umbilical points and that is holomorphic when the mean curvature is constant. Since every holomorphic quadratic differential on the Riemann sphere vanishes identically, the result follows. The arguments introduced by Hopf are specific of dimension two and they only depend on the Codazzi equation for surfaces in $\r^3$, hence Hopf's theorem is also generalized to $\h^3$ and $\s^3$.

Almost at the same time, Alexandrov \cite{ale} characterized the round sphere as the only closed and embedded CMC surface. For the proof, Alexandrov introduced the \emph{moving plane technique} that combines ambient reflections and the maximum principle of the CMC equation. Since then, it has become a fundamental tool for proving symmetry and uniqueness results for other geometric functionals and elliptic PDEs. Alexandrov's theorem easily adapts when the ambient space has reflection-type isometries and the only non-trivial issue is to guarantee that the sweeping procedure starts; in fact, it suffices the existence of two linearly-independent directions of reflection to conclude that the surface has genus zero; see \cite[Theorem 5.3.4]{dhm}. In particular, this holds in the space forms of arbitrary dimension, with the additional assumption in $\s^n$ that the surface lies contained in an open hemisphere in order to guarantee that the sweeping procedure starts. 

Later, Barbosa and do Carmo \cite{bdc} exhibited the rigidity of the round sphere under the assumption of being stable under volume-preserving variations. Their proof combines suitable test functions for the stability problem and Minkowski-type integrals. In fact, their method was also adapted to any space form \cite{bdce}.

Since the foundations of these results, two distinguished lines of inquiry appeared: either generalize them to further ambient spaces; or still formulate them in 3-dimensional space forms for the case that the CMC surface is compact but has non-empty boundary, making special emphasis in $\r^3$. We give a brief outline of the development of each line.

Besides space forms, the next most symmetric spaces are the homogeneous, simply connected 3-manifolds having 4-dimensional isometry group. They admit a Riemannian submersion onto the complete, simply connected surface of constant curvature $\kappa$ and with bundle curvature $\tau$. Each such space is denoted by $\ekt$. If $\tau=0$ the submersion trivializes to a Riemannian product; if $\tau\neq0$ we have the Heisenberg space if $\kappa=0$; the universal cover of the special linear group if $\kappa<0$; or the Berger spheres if $\kappa>0$. It is worth mentioning that a CMC sphere exists in $\ekt$ if and only if $4H^2+\kappa>0$ and thus this restriction has to be imposed if we aim to characterize a sphere. Such spheres, denoted by $S_H$, are no longer umbilical but they are rotational. Reflection-type isometries only appear\footnote{When $\kappa=4\tau^2$ we recover $\r^3$ if both vanish, or $\s^3$ with a homothetic metric. Both are limit cases.} when $\tau=0$ and Alexandrov's theorem holds in the product spaces as well, with the additional assumption when $\kappa>0$ that the projection of the surface onto $\s^2(\kappa)$ lies in a hemisphere. When $\tau\neq0$ and $\kappa\leq0$, Alexandrov's problem is still open. Regarding Barbosa-do Carmo's result, Souam \cite{sou} studied the stability of closed surfaces in the product spaces obtaining different rigidity results. When $\tau\neq0$, as far as the author knows, there is not a straightforward analog of Barbosa-do Carmo's theorem; we refer to the work \cite{tour} of Torralbo-Urbano for the stability of CMC surfaces in these homogeneous manifolds. The case of Hopf's extension has a special interest on its own. In \cite{abro1,abro2}, Abresch-Rosenberg defined a quadratic differential on every surface in $\ekt$ that is holomorphic when the mean curvature is constant, the \emph{Abresch-Rosenberg differential}, and that vanishes on every $S_H$. This differential was used to extend Hopf's result to all the $\ekt$ spaces, proving that each $S_H$ is unique among CMC immersions of a topological sphere. In general, the uniqueness of spheres in 3-dimensional manifolds has been an active field of research in the past years, culminating in some groundbreaking results. We refer the reader to \cite{dhm,dami,gamiJDG,mmpr} and references therein. 

The second line of research concerns CMC surfaces with boundary in $\r^3$. We refer the reader to \cite{loplibro} and more recently \cite{lop1} for a collection of open problems in the field. In particular, it is natural to ask for the boundary analogs of Alexandrov, Hopf and Barbosa-do Carmo. Now, the candidate surface to characterize is no longer a closed sphere, but a \emph{spherical cap}: one of the components of the round sphere delimited by a parallel. Let $\Gamma$ be a planar circle in $\r^3$ and assume that $M$ is a compact CMC surface with boundary $\Gamma$. We ask whether $M$ agrees with one of the spherical caps having $\Gamma$ as boundary in the case that: \textbf{1.} $M$ is a disk; \textbf{2.} $M$ is embedded; \textbf{3.} $M$ is stable. At the moment of the preparation of this paper, independent works of Espinar \cite{esp} and Maximo and Nunes \cite{manu} gave an affirmative answer to the first problem, which is the boundary version of Hopf's result. See also \cite{lop2} for a generalization to $\h^3$ and $\s^3$ following the ideas of \cite{manu}. The second problem is the analog of Alexandrov's theorem and it is known to hold if $M$ lies in one of the half-spaces delimited by the plane containing the boundary; in this case, Alexandrov's reflection technique can be adapted. The third problem is the analog of Barbosa-do Carmo and was known to hold when $M$ is a disk \cite{alp}.

These problems can be also stated in the $\ekt$ spaces, and not too much more is known. The first and third problems in all their generality are unknown to hold. Only when $\tau=0$ and they are combined, Zhang \cite{zha} proved that a stable CMC disk with circular boundary in $\m^2(\kappa)\times\r$ is a spherical cap of $S_H$; he also approached non-compact rotational examples. Regarding the second problem, only when $\tau=0$ and the surface has circular boundary in a horizontal plane, Alexandrov's technique can be adapted when the surface satisfies the same Euclidean hypothesis. If $\tau\neq0$, as far as the author knows, there are no results in this line.

In this paper we characterize \emph{spherical caps} of $S_H$ as the only stable immersions of a CMC disk with boundary a \emph{circle} in the Heisenberg space, the universal cover of the special linear group and some Berger spheres, that is when $\tau\neq0$. By a \emph{circle} we mean a closed simple curve invariant under ambient rotations. Specifically, we prove:

\begin{teo}\label{th}
Let $\psi:M\to\ekt$ be a stable immersion of a disk with constant mean curvature $H>0$ and circular boundary. Suppose
\begin{equation}\label{eqcondiciones}
4H^2+\kappa>0,\qquad \kappa\leq 4\tau^2.
\end{equation}
Then, $\psi$ is an embedding onto a closed spherical cap of $S_H$.
\end{teo}
It is worth mentioning that our method also recovers the conclusions achieved in \cite{alp,zha} and in particular the space $\s^2(\kappa)\times\r$, as the restriction $\kappa\leq 4\tau^2$ is only needed when $\tau\neq0$. 

Let us discuss the conditions \eqref{eqcondiciones}. The first ensures that the rotational sphere $S_H$ exists. The second only affects to certain Berger spheres and has a geometric flavor. It was proved in \cite[Theorem 1]{tour} the existence of $\alpha_0\in[0,1]$ such that for $4\tau^2/\kappa\geq\alpha_0$, CMC spheres are stable. If $\kappa\leq 4\tau^2$ we are in such hypotheses and conclude that $S_H$ is stable. Also, Torralbo \cite{tor} proved that when $\kappa\leq 4\tau^2$ then $S_H$ is embedded. Note that even if $\kappa>4\tau^2$ we could have a stable spherical cap even if $S_H$ is not stable itself.

We detail next the organization of the paper: in Section \ref{sec2}, we introduce a local coordinate model for the $\ekt$ spaces that is global if $\kappa\leq0$. We recall the fundamental data that fulfills an immersed surface in an $\ekt$ space and define the Abresch-Rosenberg differential. We also introduce the notion of stability and the Jacobi operator. In Section \ref{sec21} we prove two technical lemmas that will be useful in the sequel. The first proves that the conormal derivative of a non-identically zero Jacobi function on a stable disk cannot change of sign at the boundary of the disk too many times. The second proves that the imaginary part of a holomorphic quadratic differential on a disk changes of sign at least four times at the boundary of the disk. In Section \ref{sec3} we prove Theorem \ref{th}. If $X$ denotes the Killing vector field in $\ekt$ whose flow are the rotations and $N$ is a unit normal of the immersion, we apply Corollary \ref{corollary} to the function $u=\langle X,N\rangle$ and to a suitable modification of the Abresch-Rosenberg differential to prove that both of them vanish and hence conclude with the rotational symmetry of the immersion. Finally, some topological arguments will allow us to conclude that the immersion is actually an embedding.

\section{Preliminaries}\label{sec2}

Given $\kappa\in\r$ and $\tau>0$, let $\ekt$ be the homogeneous, simply connected 3-manifold with isometry group of dimension 4, which admits a Riemannian submersion onto the complete, simply connected surface of constant curvature $\kappa$ and with bundle curvature $\tau$. If $\pi:\ekt\to\mathbb{M}^2(\kappa)$ is the Riemannian submersion, the \emph{fibers} $\pi^{-1}(q),\ q\in\mathbb{M}^2(\kappa)$ are the integral lines of a unit Killing vector field $\xi$ and whose flow is known as the \emph{vertical translations}. Furthermore, the fibers are geodesics and when $\kappa>0$ they are closed curves. For any $X\in\mathfrak{X}(\ekt)$ we have
$$
\onabla_X\xi=\tau X\wedge\xi,
$$ 
where $\onabla$ is the Levi-Civita connection of $\ekt$. In the sequel, after a change of the orientation we assume $\tau>0$. Furthermore, we discard the limit case $\kappa=4\tau^2$ since otherwise we are in a space form.

A local coordinate model is the following: let $r=\sqrt{x^2+y^2}$ and define $\delta(r)=1+\frac{\kappa}{4}r^2$. Then, $\ekt$ is locally isometric to the space
\begin{equation}\label{eqmetricaEkt}
\rkt=\{(x,y,z)\in\r^3\colon \delta(r)>0\},\qquad \langle,\rangle=\frac{dx^2+dy^2}{\delta(r)^2}+\left(dz+\frac{\tau}{\delta(r)}(ydx-xdy)\right)^2.
\end{equation}
When $\kappa\leq0$ this model is global, and if $\kappa>0$ the model is the universal covering of the space minus one fiber. The vector fields
$$
E_1=\delta(r)\partial_x-\tau y\xi,\quad E_2=\delta(r)\partial_y+\tau x\xi,\quad E_3=\partial_z=\xi,
$$
are a global orthonormal frame of $\rkt$. Given $R>0$ with $\delta(R)>0$, we define the coordinate circle $\Gamma_R$ in $\{z=0\}$ as the image of the curve
$$
\gamma_R(t)=(R\cos t,R\sin t,0),\qquad t\in[0,2\pi].
$$

We recall the fundamental data associated to an isometric immersion $\psi:M\to\ekt$. Consider $\z=s+it$ a conformal parameter for the conformal structure given by the induced metric and let $\partial_\z=(\partial_s-i\partial_t)/2,\ \partial_{\bar{\z}}=(\partial_s+i\partial_t)/2$. Denote by $N$ to a unit normal to the immersion, $S$ the shape operator, $\sigma$ the second fundamental form and $J$ the $\pi/2$ rotation in $TM$, i.e.
$$
Su=-\onabla_u N,\quad \sigma(u,v)=\langle Su,v\rangle,\quad Ju=N\wedge u.
$$
The \emph{fundamental data} of $\psi$ is the 5-tuple $(\lambda^2|d\z|^2,\nu,H,p\,d\z^2,A\,d\z)$, where $H$ is the mean curvature and
$$
\lambda^2=2\langle\psi_\z,\psi_{\bar{\z}}\rangle,\quad \nu=\langle N,\xi\rangle,\quad p=\langle\onabla_{\partial_\z}\psi_\z,N\rangle,\quad A=\langle\psi_\z,\xi\rangle.
$$
The function $\nu$ is commonly known as the \emph{angle function} and $p\,d\z^2$ is the Hopf differential. The integrability equations in $\ekt$ are expressed in terms of these fundamental data. See \cite{dan,femiICM} for details. A fundamental tool in the study of CMC surfaces in the $\ekt$ spaces is the Abresch-Rosenberg differential, defined by
\begin{equation}\label{eqQab}
\Q d\z^2=\left(2(H+i\tau)p-(\kappa-4\tau^2)A^2\right)d\z^2.
\end{equation}
Using the compatibility equations of \cite[Theorem 3.2]{femiICM}, one can see that
$$
\Q_{\bar{\z}}=2pH_{\bar{\z}}+\lambda^2(H+i\tau)H_\z,
$$
and hence $\Q$ is holomorphic on any CMC surface. Furthermore, $\Q$ is known to vanish on every $S_H$ but the converse is not true; the situation here is that if $\Q$ vanishes identically on a CMC surface, then such surface is invariant by a continuous 1-parameter group of isometries of $\ekt$, see \cite[Lemma 2.3.6]{dhm}. This heavily contrast with the situation in $\r^3$, where the vanishing of the Hopf differential yields total umbilicity.

Now, we recall the notion of stability. Given $M$ a CMC surface in $\ekt$, it is known that $M$ is a critical point of the area functional for compactly supported variations $\{M_t\colon t\in(-\epsilon,\epsilon),\ M_0\equiv M\}$ that fix both $\partial M$ and the volume. If $\mathcal{A}(t)$ is the area of $M_t$, then $M$ is said to be \emph{stable} if $\mathcal{A}''(0)\geq0$ for every such variation. Stability of $M$ is equivalent to
$$
\mathcal{A}''(0)=-\int_M u(\Delta u+(|\sigma|^2+\ricn )u)\geq0,
$$
for every $u\in C_0^\infty(M)$ with $\int_Mu=0$. Here, $\sigma$ is the second fundamental form of $M$ and $\ricn$ is the Ricci curvature along the direction of the unit normal $N$ of $M$. The Jacobi operator is defined then by $L=\Delta+|\sigma|^2+\ricn$, which is an elliptic, self-adjoint operator that takes the form
$$
L=\Delta+|\sigma|^2+\kappa-2\tau^2-(\kappa-4\tau^2)\nu^2=\Delta+\textbf{q}.
$$
Integration by parts allows us to define the bilinear form
$$
I(u,v)=\int_M\langle\nabla u,\nabla v\rangle-\textbf{q}uv,
$$
where this time $u,v$ belong to the first Sobolev space $H_0^1(M)$. Thus, stability means
\begin{equation}\label{eqoperadorestable}
I(u,u)\geq0,\quad\mbox{for all $u\in H_0^1(M)$ such that $\int_M u=0$}.
\end{equation}
The ellipticity and self-adjointness of $-L$ imply that the Dirichlet problem
$$
\left\lbrace
\begin{array}{lll}
Lu+\lambda u=0&\mathrm{in}&M,\\
u=0&\mathrm{on}&\partial M
\end{array}
\right.
$$
has as \emph{spectrum} an ordered sequence of eigenvalues
$$
\lambda_1<\lambda_{2,1}=\cdots=\lambda_{2,m_2}<\lambda_{3,1}=\cdots=\lambda_{3,m_3}<\cdots\nearrow\infty
$$
being always the first simple; here, the subindex $m_k$ stands for the multiplicity of each ordered distinct eigenvalue $\lambda_k$. The \emph{index} of $M$ is the number of negative eigenvalues of $-L$, counting multiplicity. It is well-known that the existence of two negative eigenvalues forbids stability: we could define a zero mean function by a linear combination of two eigenfunctions where $I$ is negative. Given an eigenfunction $u_k$ of an eigenvalue $\lambda_k\equiv\lambda_{k,1}=\cdots=\lambda_{k,m_k}$, the nodal set of $u_k$ is $u_k^{-1}(0)$, and the nodal domains of $u_k$ are the connected components of $M\backslash u_k^{-1}(0)$. Courant's nodal domain theorem \cite{cohi} states that the number of nodal domains of $u_k$ is at most $1+\sum_{j=1}^{k-1}m_j$, being always $m_1=1$.

\subsection{Two technical lemmas.}\label{sec21}

Next we introduce two technical lemmas that will be of importance to prove Theorem \ref{th}. The first one states that the conormal derivative of a non-trivial Jacobi function of a stable disk cannot alternate of sign on the boundary many times.

\begin{lem}\label{lemderivadaconormal}
Let be $M$ a disk and $u\not\equiv0$ a solution of $Lu=0$ with $u_{|_{\partial M}}=0$ and $L$ satisfying \eqref{eqoperadorestable}. If $\mu$ is the unit exterior conormal of $M$ along $\partial M$, then $\partial_\mu u$ cannot have four alternating boundary signs.
\end{lem}

\begin{proof}
Let us uniformize $M$ by $\overline{\D}$. Arguing by contradiction, assume that $\partial_\mu u$ alternates of sign at $\xi_1,...,\xi_4\in\partial M$. Now, we find four disjoint sets $\Omega_1,...,\Omega_4$ around $\xi_1,...,\xi_4$ respectively such that $u$ has the opposite sign of $\partial_\mu u$. This can be done by taking interior normal coordinates around $\xi_k$ and using that $u=0$ at $\partial\D$. Indeed, given $e^{i\vartheta}\in\partial\D$ we consider normal coordinates
$$
\Psi(\vartheta,r)=\exp_{e^{i\vartheta}}(-r\mu(e^{i\vartheta})),
$$
where $r\geq0$ since $\mu$ is the exterior conormal. Thus, $\partial_r\Psi(\vartheta,0)=-\mu(e^{i\vartheta})$. Define the function $U(\vartheta,r)=u(\Psi(\vartheta,r))$. One has $U(\vartheta,0)=0$. For fixed $\vartheta$ we make Taylor in the $r$-variable, obtaining
$$
U(\vartheta,r)=U(\vartheta,0)+r\partial_r U(\vartheta,0)+\mathcal{O}(r^2).
$$
The first term vanishes and for the second one we obtain $-\partial_\mu u(e^{i\vartheta})$ by the chain rule. Thus,
$$
u(\Psi(\vartheta,r))=-r\partial_\mu u(e^{i\vartheta})+\mathcal{O}(r^2).
$$
By continuity, there exists a compact arc $\mathcal{J}$ containing $p$ such that the term $\mathcal{O}(r^2)$ can be uniformly bounded by $Cr^2$, where $C>0$ and thus the sign of $u$ is opposite to the one of $\partial_\mu u$ at $\mathcal{J}$.

We state that $u$ has at least three nodal domains. Indeed, if $u$ has only two nodal domains, the alternating sets $\Omega_1,\Omega_3$ where $u$ has the same sign should belong to the same connected component. But this configuration would make impossible to join the sets $\Omega_2$ and $\Omega_4$ where $u$ has the same sign and opposite to the one of $\Omega_1\cup\Omega_3$.  This contradiction proves that $u$ has at least three nodal domains and hence Courant's nodal domain theorem yields that the zero eigenvalue is at least the third distinct one of $-L$. This implies the existence of two negative eigenvalues of $-L$, contradicting the stability.

\end{proof}

The next result proves that the imaginary part of a holomorphic quadratic differential on a disk must alternate of sign a certain number of times.

\begin{lem}\label{lemmaholomorfa}
Let be $M$ a disk and $\Phi$ a holomorphic quadratic differential on $M$, smooth up to $\partial M$. Then, either $\Phi\equiv0$ or $\im\Phi(\t,\t)$ changes of sign at least four times on $\partial M$, where $\t$ is a unit boundary tangent.
\end{lem}

\begin{proof}
We uniformize $M$ by $\bar{\D}$, with coordinates $\z=s+it$, write $\langle,\rangle=\lambda^2|d\z|^2$ and express $\Phi=\phi(\z)d\z^2$. If we parametrize the boundary $|\z|=1$ by $\beta(\vartheta)=e^{i\vartheta}$, the tangent vector $\beta'$ has length $\lambda$ and thus $\t=\pm\frac{1}{\lambda}\beta'(\vartheta)=\frac{\pm i}{\lambda}\beta(\vartheta)$. Consequently, $d\z(\t)=(ds+idt)(\frac{\pm i}{\lambda}e^{i\vartheta})=\frac{\pm i}{\lambda}\z$ and therefore 
$$
\Phi(\t,\t)=\phi(\z)d\z^2(\t,\t)=\phi(\z)\frac{-\z^2}{\lambda^2}.
$$
In particular, along $\partial\D$ we have
$$
\im\Phi(\t,\t)=-\frac{1}{\lambda^2}\im(\z^2\phi(\z)).
$$
By defining $F(\z)=\z^2\phi(\z)$ we conclude $\mathrm{Im}F=-\lambda^2 \im\Phi(\t,\t)$ on $\partial\D$.

Assume that $\Phi\not\equiv 0$ and let us consider $w(\vartheta)=\im F(e^{i\vartheta})$. We prove the existence of $\vartheta_1,...,\vartheta_4\in[0,2\pi)$ such that $w(\vartheta_1),...,w(\vartheta_4)$ have alternating signs. For that, bear in mind that the holomorphic function has, at least, an order two zero at $\z=0$, i.e. $F(0)=F'(0)=0$ and thus its Taylor series is  $F(\z)=\sum_{n=2}^\infty c_n\z^n$, on a disk $D_r$ of radius $r<1$. Thus, at its boundary $\z=re^{i\vartheta}$ and denoting $c_n=a_n+ib_n$ we have
$$
\im F(r e^{i\vartheta})=\sum_{n=2}r^n(a_n\sin(n\vartheta)+b_n\cos(n\vartheta)).
$$
Set $w_r(\vartheta)=\im F(re^{i\vartheta})$. By the orthogonality of the functions $\{1,\cos(n\vartheta),\sin(n\vartheta)\}$ in $L^2(0,2\pi)$ we have
$$
\int_0^{2\pi}w_r(\vartheta)\{1,\cos\vartheta,\sin\vartheta\}\,d\vartheta=0.
$$
Now, since $F$ is continuous up to $\partial\D$ and $w_r\to w$ uniformly, letting $r\nearrow1$ we conclude
\begin{equation}\label{eqalternanciaceros}
\int_0^{2\pi}w(\vartheta)\,d\vartheta=\int_0^{2\pi}w(\vartheta)\cos\vartheta\,d\vartheta=\int_0^{2\pi}w(\vartheta)\sin\vartheta\,d\vartheta=0.
\end{equation}
The first equation in \eqref{eqalternanciaceros} yields the existence of $\vartheta_1,\vartheta_2\in[0,2\pi)$ such that $w$ alternates of sign; without losing generality assume $w(\vartheta_1)>0$ and $w(\vartheta_2)<0$. In particular, there exists $n_1\in(\vartheta_1,\vartheta_2)$ such that $w(n_1)=0$. This cannot be the final configuration for the alternating signs of $w$ since $w(0)=w(2\pi)$ but $w_{|_{[0,\vartheta_1]}}>0$ and $w_{|_{[\vartheta_2,2\pi]}}<0$. Thus, there must exist $n_2\in(0,\vartheta_1)$ such that $w(n_2)=0$ and $w<0$ for $\vartheta<n_2$ close enough to $n_2$. This gives three alternating signs of $w$. If there were no four alternating points, a compact identification of $[0,2\pi]$ with $\partial\D$ would give two points $\xi_1=e^{i n_1}$ and $\xi_2=e^{i n_2}$ that separate $\partial\D$ into two complementary arcs where $\im F$ take different signs. Let $\vartheta\in[n_2,n_2+2\pi]$ and define
$$
P(\vartheta)=\cos\left(\vartheta-\frac{n_1+n_2}{2}\right)-\cos\left(\frac{n_1-n_2}{2}\right).
$$
If $\vartheta\in(n_2,n_1)$ then $P(\vartheta)>0$, and if $\vartheta\in(n_1,n_2+2\pi)$ then $P(\vartheta)<0$, with $P(n_1)=P(n_2)=0$. After reversing the sign if necessary we have $wP\geq0$ everywhere and $wP>0$ on a non-empty open interval. Thus,
$$
\int_0^{2\pi}w(\vartheta)P(\vartheta)\,d\vartheta>0.
$$
But this contradicts \eqref{eqalternanciaceros} by just evaluating the expression of $P$.  

\end{proof}

Combining Lemmas \ref{lemderivadaconormal} and \ref{lemmaholomorfa}, we obtain.

\begin{cor}\label{corollary}
With the notation of Lemmas \ref{lemderivadaconormal}, \ref{lemmaholomorfa}, assume that on $\partial M$ we have
$$
\im\Phi(\t,\t)=c\partial_\mu u,
$$
where $c>0$ is smooth. Then, $\Phi\equiv0$ and $u\equiv0$.
\end{cor}

\section{Proof of Theorem \ref{th}}\label{sec3}

This section is devoted to prove Theorem \ref{th}. The idea is to apply Corollary \ref{corollary} to a suitable holomorphic quadratic differential and a Jacobi function that carry geometric information when they both vanish. It will surprise no one that the holomorphic quadratic differential has to be related with the Abresch-Rosenberg differential. As for the function $u$, we consider the Killing vector field $X=-y\partial_x+x\partial_y$, whose flow are the rotations about the $z$-axis in the coordinate model \eqref{eqmetricaEkt}, and define $u=\langle X,N\rangle$. Thus, $u\equiv0$ on a surface means that $X$ is everywhere tangent to it and hence rotational.

For our purpose, it is better to introduce polar coordinates $x=r\cos\theta,\ y=r\sin\theta$. The metric \eqref{eqmetricaEkt} is expressed by
$$
\langle,\rangle=\frac{dr^2+r^2d\theta^2}{\delta(r)^2}+\left(dz^2-\frac{\tau r^2}{\delta(r)}d\theta\right)^2.
$$
Thus, in this model $X=\partial_\theta$. We drop the dependence of $\delta$ on $r$ and define the orthonormal frame
$$
e_r=\delta\partial_r,\qquad e_\theta=\frac{\delta}{r}\partial_\theta+\tau r\xi,\qquad \xi=\partial_z.
$$
By defining $\beta_r=\frac{4-\kappa r^2}{4r}$, a straightforward computation yields
$$
[e_r,e_\theta]=-\beta_r e_\theta+2\tau\xi,\qquad [e_r,\xi]=[e_\theta,\xi]=0.
$$
The connection matrix is
$$
\begin{array}{lll}
\overline{\nabla}_{e_r}e_r=0&\overline{\nabla}_{e_r}e_\theta=\tau\xi&\overline{\nabla}_{e_r}\xi=-\tau e_\theta\\
\overline{\nabla}_{e_\theta}e_r=\beta_re_\theta-\tau\xi&\overline{\nabla}_{e_\theta}e_\theta=-\beta_re_r&\overline{\nabla}_{e_\theta}\xi=\tau e_r\\
\overline{\nabla}_{\xi}e_r=-\tau e_\theta&\overline{\nabla}_{\xi}e_\theta=\tau e_r&\overline{\nabla}_{\xi}\xi=0
\end{array}
$$
Geometrically, the parameter $\beta_r$ is the curvature of the projected circle in the base. From now on, the coordinate radius of the circle is $R>0$ and 
$$
\delta_R=1+\frac{\kappa R^2}{4},\qquad W=\sqrt{1+\tau^2 R^2}.
$$
Along the circle $\gamma(\theta)=(R\cos\theta,R\sin\theta,0)$, we have $\gamma'(\theta)=\partial_\theta=\frac{R}{\delta_R}(e_\theta-\tau R\xi)$, which has length $RW/\delta_R$. Hence,
$$
\t=\frac{e_\theta-\tau R\xi}{W}
$$
is a unit tangent of $\gamma$ in the orthonormal frame $\{e_r,e_\theta,\xi\}$. Also, since the tangent plane of the section $\{z=0\}$ is spanned by the vectors $\partial_r$ and $\t$, the vector
$$
\eta=\frac{\tau R e_\theta+\xi}{W}
$$
is the upwards unit normal of $\{z=0\}$ along $\Gamma_R$, and $\{e_r,\t,\eta\}$ is a positively oriented orthonormal frame. In particular, the unit normal $N$ along $\Gamma_R$ satisfies
\begin{equation}\label{eqNab}
N=ae_r+b\eta,\qquad a^2+b^2=1,\qquad \n=J\t=a\eta-b e_r.
\end{equation}
As a matter of fact, the exterior unit conormal $\mu$ along $\Gamma_R$ satisfies $\mu=\epsilon\n$ with $\epsilon\in\{1,-1\}$. We can choose a fixed $\epsilon$ since $\partial M$ is connected. 

\subsection{A priori flux bound.}\label{sec31}

The following result follows from the CMC flux and the particular symmetry of the boundary circle.

\begin{lem}\label{fluxlema}
With the previous notation, the following identity holds:
\begin{equation}\label{eqflux}
\frac{1}{2\pi d}\int_0^{2\pi d}a(\theta)\,d\theta=-HR,
\end{equation}
where $d$ is the number of times that $\partial M$ is mapped onto $\Gamma_R$. Furthermore, $HR\leq 1$ and equality holds if and only if $a\equiv -1$ and thus $b=0$ in \eqref{eqNab}.
\end{lem}

\begin{proof}

Let $\Omega\subset\{z=0\}$ be disk bounded by $\Gamma_R$ with unit normal $\eta$. In particular, the positive orientation of $\partial\Omega$ is $\t$, since $\{e_r,\t,\eta\}$ is positive. On $\partial M$ we have the positive frame $\{\t,\n\}$. We also ask $\{\mu,\pm\t\}$ to be positively oriented. Because $\mu=\epsilon\n$ and $\{\n,-\t\}$ has the same orientation of $\{\t,\n\}$, then $\{\epsilon\n,-\epsilon\t\}$ is positive. Hence, the orientation of $\partial M$ is determined by $-\epsilon\t$. We consider thus the 2-cycle
$$
Z=[\psi(M)]+\epsilon d [\Omega],
$$
where $d\geq1$ is the number of times that $\Gamma_R$ is covered by $\partial M$. This cycle bounds a 3-chain whose boundary is $Z$, and the divergence theorem to $\xi$ on such a 3-chain yields
$$
0=\int_M\langle N,\xi\rangle+\epsilon d\int_\Omega\langle\eta,\xi\rangle.
$$
For the second integral, recall that in the polar coordinates $(r,\theta)$, the area element in $\{z=0\}$ is $\frac{rW_r}{\delta^2}dr\,d\theta$; here, $W_r=\sqrt{1+\tau^2r^2}$. Hence,
$$
\int_\Omega\langle\eta,\xi\rangle=\int_0^{2\pi}\int_0^R\frac{r}{\delta^2}\,dr\,d\theta=\frac{\pi R^2}{\delta_R}.
$$
We integrate again on $M$, but this time the tangent component $\xi^\top$. Since $\xi$ is Killing, $\div_M\xi^\top=2H\langle\xi,N\rangle$ and the line element of the circle is $d\ell=|\gamma'(\theta)|d\theta=\frac{RW}{\delta_R}d\theta$, the divergence theorem on $M$ yields
$$
2H\int_M\nu=\int_M2H\langle \xi,N\rangle=\int_{\partial M}\langle \xi,\mu\rangle\,d\ell=\frac{\epsilon}{W}\int_{\partial M}a\,d\ell=\frac{\epsilon R}{\delta_R}\int_0^{2\pi d}a(\theta)\,d\theta.
$$
We conclude
$$
\frac{1}{2\pi d}\int_0^{2\pi d}a(\theta)\,d\theta=-HR.
$$
As $|a|\leq1$ we obtain the bound $HR\leq1$ and equality holds if and only if $a=-1$ and then $b=0$.

\end{proof}

\begin{remark}\label{remarkrelacionHR}
This relation gives an a priori restriction between the coordinate radius $R$ and the possible mean curvature $H$ of a compact CMC surface with boundary. Although it will not be used throughout the paper, it would be better to express this restriction in terms of the distance of $\Gamma_R$ to the rotation axis, hence in terms of the distance that a parallel in $S_H$ can attain to the rotation axis. Let us fix $(R,\theta_0,0)\in\Gamma_R$ and consider the radial geodesic $r\mapsto(r,\theta_0,0)$. Thus,
$$
\rho(R)=\mathrm{dist}_{\ekt}(\Gamma_R,(0,0,0))=\int_0^R\frac{dr}{\delta(r)}=\int_0^R\frac{4\,dr}{4+\kappa r^2}=
\left\lbrace
\begin{array}{lc}
\vspace{.1cm}\dfrac{2}{\sqrt{\kappa}}\arctan\left(\dfrac{\sqrt{\kappa}R}{2}\right)&\kappa>0,\\
R,&\kappa=0,\\
\dfrac{2}{\sqrt{-\kappa}}\mathrm{arctanh}\left(\dfrac{\sqrt{-\kappa}R}{2}\right)&\kappa<0.
\end{array}
\right.
$$
Since $HR\leq1$ and $R\mapsto\rho(R)$ is strictly increasing in any case, the maximum distance to the rotation axis that $S_H$ attains is $r_H=\rho(1/H)$, that is
$$
r_H=\left\lbrace
\begin{array}{lc}
\vspace{.1cm}\dfrac{2}{\sqrt{\kappa}}\arctan\left(\dfrac{\sqrt{\kappa}}{2H}\right)&\kappa>0,\\
\dfrac{1}{H},&\kappa=0,\\
\dfrac{2}{\sqrt{-\kappa}}\mathrm{arctanh}\left(\dfrac{\sqrt{-\kappa}}{2H}\right)&\kappa<0.
\end{array}
\right.
$$
Moreover, if equality holds then $a=-1$ and $b=0$ in \eqref{eqNab} and hence the tangent plane of $S_H$ along $\Gamma_R$ is spanned by $\{\t,\xi\}$ and consequently it is vertical, therefore $\Gamma_R$ is the \emph{equator} of $S_H$, the parallel of $S_H$ with largest distance to the rotation axis, and $S_H$ is a symmetric bi-graph along $\Gamma_{1/H}$.
\end{remark}

\subsection{The expressions of $\partial_\mu u$ and $\Q d\z^2$ along $\partial M$.}\label{sec32}

As stated in the beginning of the section, we aim to apply Corollary \ref{corollary} and for that matter we compute both $\partial_\mu u$ and $\Q d\z^2$ along $\partial M$. We do the first next.

\begin{pro}\label{propderconormalu}
Let us denote by $k_R$ and $-\uptau_R$ the geodesic curvature and the torsion of the circle $\Gamma_R$ in $\{z=0\}$, respectively. On $\partial M$ we have
\begin{equation}\label{eqderconormalu}
\partial_\mu u=-\epsilon\frac{RW}{\delta_R}\left(\uptau_R+\sigma(\t,\n)\right)=\epsilon\frac{RW}{\delta_R}\alpha',
\end{equation}
where $\alpha$ is such that $a=\cos\alpha$ and $b=\sin\alpha$.
\end{pro}

\begin{proof}
Recall that
$$
\partial_\mu u=\mu\langle X,N\rangle=\langle\onabla_\mu X,N\rangle+\langle X,\onabla_\mu N\rangle=\langle\onabla_\mu X,N\rangle-\langle X,S\mu\rangle.
$$
We compute first the second term. Since $\mu=\epsilon\n$ and along the boundary $X=\frac{RW}{\delta_R}\t$, we have
$$
\langle X,S\mu\rangle=\epsilon\frac{RW}{\delta_R}\langle\t,S\n\rangle=\epsilon\frac{RW}{\delta_R}\sigma(\t,\n).
$$
To compute the first term, it suffices to compute $\onabla_\n X$. Since $[X,e_r]=[X,\eta]=0$, the following commutation holds:
$$
\onabla_{e_r}X=\onabla_X e_r,\qquad \onabla_\eta X=\onabla_X\eta.
$$
Now, we use that $\n=a\eta-b e_r$, hence
$$
\onabla_\n X=a\onabla_\eta X-b\onabla_{e_r}X=a\onabla_X\eta-b\onabla_Xe_r.
$$
Evaluating on $\partial M$,
$$
\onabla_\n X=\frac{RW}{\delta_R}\left(a\onabla_\t\eta-b\onabla_\t e_r\right).
$$
Now we use the Darboux equations
$$
\onabla_\t e_r=k_R\t+\uptau_R\eta,\quad\onabla_\t\eta=-\uptau_Re_r,\qquad k_R=\frac{1}{R}+\frac{(4\tau^2-\kappa)R}{4W^2},\ \uptau_R=\frac{(4\tau^2-\kappa)\tau R^2}{4W^2}.
$$
Geometrically, $k_R$ is the curvature of the circle $\Gamma_R$ in $\{z=0\}$ and it must not be misunderstood with $\beta_R$ that was defined as the curvature of the projected circle in the base; they only agree if $\tau=0$. In fact, $k_R=\frac{\beta_R+2\tau^2R}{W^2}$.
Also, $-\uptau_R$ is the torsion of $\Gamma_R$. With these equations,
$$
\onabla_\n X=\frac{RW}{\delta_R}\left(a\onabla_\t\eta-b\onabla_\t e_r\right)=\frac{RW}{\delta_R}\left(-a\uptau_R e_r-bk_R\t-b\uptau_R\eta\right).
$$
Taking product with $N=ae_r+b\eta$ yields
$$
\langle\onabla_\n X,N\rangle=-\frac{RW}{\delta_R}\uptau_R.
$$
Finally, since $N=a e_r+b\eta$ with $a^2+b^2=1$, there exists $\alpha$ such that
$$
a=\cos\alpha,\qquad b=\sin\alpha.
$$
The angle $\alpha$ is defined locally, but its derivative $\alpha'=t(\alpha)=\frac{d\alpha}{ds}$ is global; indeed $a'=-b\alpha'$, $b'=a\alpha'$ and thus $\alpha'=ab'-ba'$. A direct computation yields
$$
\onabla_\t N=a k_R\t+(\alpha'+\uptau_R)\n.
$$
Consequently, $\sigma(\n,\t)=-\alpha'-\uptau_R$ and we conclude with \eqref{eqderconormalu}.

\end{proof}

Our next step is to relate $\Q d\z^2(\t,\t)$ with the expression derived in Proposition \ref{propderconormalu}. To save notation we write $\Q(\t,\t)\equiv \Q d\z^2(\t,\t)$, although we emphasize that only the latter can be evaluated on $(\t,\t)$. An evaluation yields
\begin{align*}
\Q(\t,\t)&=(H+i\tau)(\sigma(\t,\t)-H-i\sigma(\t,\n))+\left(\tau^2-\frac{\kappa}{4}\right)(\langle\xi,\t\rangle-i\langle\xi,\n\rangle)^2\\
&=(H+i\tau)(-k_Ra-H+i(\alpha'+\uptau_R))+\frac{4\tau^2-\kappa}{4W^2}(\tau R+ia)^2
\end{align*}
We collect the terms in a suitable way and write
$$
\Q(\t,\t)=(-\tau+iH)\alpha'-\left(H+\frac{a}{R}\right)(D+iE),\qquad D=H+\frac{(4\tau^2-\kappa)Ra}{4W^2},\ E=\frac{\tau\delta_R}{W^2}.
$$
We aim to isolate the term of $\alpha'$ in the imaginary part of $\Q(\t,\t)$ and for that matter let $m=D-iE$ and observe that
$$
\im(m\Q(\t,\t))=\mathcal{B}\alpha',\qquad \mathcal{B}=\frac{4H^2+\kappa}{4}+\frac{4\tau^2-\kappa}{4W^2}(1+HRa).
$$
At this point, we remark that the coefficient $\mathcal{B}$ is strictly positive on $\partial M$ because $|a|\leq1$, the flux \eqref{eqflux} and the assumptions \eqref{eqcondiciones}. This expression of the imaginary part of $m\Q(\t,\t)$ only holds on $\partial M$ since $D$ is only defined on $\partial M$, and in order to apply Lemma \ref{lemmaholomorfa} we would need $m\Q$ to be defined on the whole disk. We construct a suitable extension next. The function $m=D-iE$ lies in the lower half-plane because $E>0$; in particular, there is a global choice of an argument $\varphi$ of $m$ along the whole $\partial M$, that is $m=|m|e^{i\varphi}$ with $\varphi\in(-\pi,0)$. Now we take advantage that $M$ is a disk by taking the harmonic extension of $\varphi$ to the whole $M$, that is, we consider the unique solution to
$$
\left\lbrace
\begin{array}{rll}
\Delta v=0&\mathrm{on}& M,\\
v=\varphi&\mathrm{on}&\partial M.
\end{array}
\right.
$$
Now, let $h$ be a harmonic conjugate of $v$, i.e. such that $h+iv$ is a holomorphic function. Finally, we define $G=e^he^{iv}$, which is holomorphic, does not vanish on $M$ and on $\partial M$ we have $G_{|_{\partial M}}=e^he^{i\varphi}=\frac{e^h}{|m|}m$. With all this setting, let us define $\Phi=\epsilon G\Q d\z^2$, which is a holomorphic differential on $M$ that satisfies
$$
\im\Phi(\t,\t)=\epsilon\frac{e^h}{|m|}\mathcal{B}\alpha'=\frac{e^h}{|m|}\mathcal{B}\frac{\delta_R}{RW}\partial_\mu u=c\partial_\mu u,\qquad c=\frac{e^h\delta_R \mathcal{B}}{RW|m|}>0.
$$
By Corollary \ref{corollary} we conclude $\Phi\equiv0$ and since $G$ never vanishes we obtain $\Q\equiv 0$ and $\langle X,N\rangle\equiv0$. This implies that $\psi(M)$ is invariant under the rotations pointwise fixing the $z$-axis.

\begin{remark}
If $\tau=0$ then $E=0$ and the aforementioned argument fails. Nonetheless, for this case we have $W=1$ and
$$
\im\Q(\t,\t)=H\alpha',\qquad \partial_\mu u=\epsilon\frac{R}{\delta_R}\alpha'.
$$
Thus, it suffices to take $\Phi=\epsilon\Q$ and $c=\frac{H\delta_R}{R}$ to conclude. 
\end{remark}

At this point, we only know that $\psi(M)$ is rotational but do not know if $\psi(M)\subset S_H$, or even in such a case if $\psi(M)$ lies contained in a spherical cap of $S_H$ or whether $\psi$ parametrizes that cap diffeomorphically. We answer these questions next to conclude with the proof of Theorem \ref{th}.

We begin by proving that $\psi(M)$ necessarily intersects the rotation axis. Indeed, consider the vertical projection
$$
\Pi\circ\psi:M\to\r^2,\qquad \Pi\circ \psi(p)=(x(\psi(p)),y(\psi(p)),
$$
where by $x(\psi),y(\psi)$ we mean the first and second coordinate, respectively, of $\psi$ in the model $\rkt$. If $\psi(M)$ does not intersect the rotation axis, then $(0,0)\notin\Pi\circ\psi(M)$. By compactness, there exist $0<R_1<R_2$ such that $\Pi\circ\psi(M)$ lies in the annulus 
$$
\{(x,y)\in\r^2\colon R_1^2\leq x^2+y^2\leq R_2^2\}.
$$
Now, consider the restriction $\Pi\circ\psi_{|_{\partial M}}:\partial M\to\Gamma_R$. Since $\psi$ is an immersion onto $\Gamma_R$, it is a covering map of a certain degree $d\neq0$; in fact, $\Pi\circ\psi_{|_{\partial M}}$ maps $\partial M$ as a loop twisting around the origin a number of times $|d|$. On the other hand, $\Pi\circ\psi_{|_{\partial M}}$ was defined as a restriction to $\partial M$ of $\Pi\circ\psi$ defined on $M$, and every loop that extends to a disk is null-homologous, yielding the desired contradiction. 

Now, we prove that $\psi(M)\subset S_H$. Let be $\textbf{o}\in\mathrm{int}(M)$ such that $\psi(\textbf{o})$ lies in the rotation axis. Since $\psi(M)$ is rotational then the unit normal at $\psi(\textbf{o})$ is $\pm\xi$. In particular, $\psi(M)$ is a radial graph near $\psi(\textbf{o})$ and the uniqueness of the Dirichlet problem with zero boundary data and an analytic prolongation argument allows us to conclude that $\psi(M)\subset S_H$.

Next, we exhibit that $\psi(M)$ actually lies contained in one of the spherical caps determined by $\Gamma_R$ in $S_H$. Let us decompose $S_H$ by 
$$
S_H=S_H^+\cup S_H^-\cup\Gamma_R.
$$
Let be $\mathcal{C}=\psi^{-1}(\Gamma_R)$. Since $\psi$ is an immersion in particular it is a local diffeomorphism, hence for every $q\in\Gamma_R$ there exists $p\in\partial M$ such that a neighborhood $V$ of $p$ is mapped diffeomorphically to a neighborhood of $q$. Since $V$ is topologically a half-disk with boundary at $\partial M$ and by compactness, we can choose a neighborhood $V$ of $\partial M$ such that $\mathcal{C}\cap V=\partial M$. Thus, $V\backslash\partial M$ is mapped by $\psi$ to one of the spherical caps, say $S_H^+$. Assume by contradiction that $\psi(M)$ has also points in $S_H^-$. In particular, $\mathcal{C}$ has points in $\mathrm{int}(M)$. Because $\psi:\mathrm{int}(M)\to S_H$ is a local diffeomorphism and $\mathcal{C}\backslash\partial M$ is compact since it lies in $M\backslash V$ which is compact,
$$
\mathcal{C}=\partial M\cup C_1\cup\cdots\cup C_k,
$$
where each $C_j$ is a closed simple curve that bounds a disk $D_j$. Define
$$
U=M\backslash\bigcup_{j=1}^k\mathrm{int}(D_j).
$$
If some $D_i\subset D_j$ then we just take away the whole $D_j$. In this scenario, $U$ is closed and connected with $\partial M\subset U$ and such that no point of $\mathrm{int}(U)$ maps via $\psi$ to $\Gamma_R$. Since $V\subset U$, necessarily $\psi(\mathrm{int}(U))\subset S_H^+$. In this case, $\psi:\mathrm{int}(U)\to S_H^+$ is a local diffeomorphism and \emph{proper}; indeed, if $K\subset S_H^+$ is compact, then its preimage in $U$ is compact and does not intersect $\partial U$, hence is compact in $\mathrm{int}(U)$. A proper local diffeomorphism is both open and closed, and since $S_H^+$ is connected then $\psi$ is surjective and hence a covering map. Since $S_H^+$ is simply connected, the covering has a single sheet, proving that actually $\psi:\mathrm{int}(U)\to S_H^+$ is a diffeomorphism. This is however a contradiction since $\mathrm{int}(U)$ has holes in it and $S_H^+$ is simply connected. This contradiction proves that $\psi:\mathrm{int}(M)\to S_H^+$ is a diffeomorphism. Finally, this also extends to the boundary since otherwise there would exist $p_1,p_2\in\partial M$ such that $\psi(p_1)=\psi(p_2)\in\Gamma_R$. Thus, we would find $V_1,V_2$ neighborhoods of $p_1,p_2$ respectively such that $\psi(V_1)\cap\psi(V_2)\neq\varnothing$ in $S_H^+$, contradicting that $\psi$ is a diffeomorphism in $\mathrm{int}(M)$.

In the case of the Berger spheres, $\kappa\leq 4\tau^2$ implies that $S_H$ is embedded \cite{tor} and hence $\Gamma_R$ also separates $S_H$ into two caps and the argument holds.

\subsection*{Ethics declarations}

The author has no conflict of interest to declare that are relevant to the content of this article. No data were used to support this study.

\subsection*{AI usage}

AI was used to thoroughly check all the boundary computations, a task performed by the author in other occasions with some specialized software, and for validate some arguments. Anyway, the author takes full responsibility for the final content of the manuscript.

%%%%%%
\subsection*{Acknowledgment}

Partially supported by the grant PID2025-167696NB-I00, funded by MICIU/AEI/ 10.13039/501100011033.

Universidad de Murcia. \emph{E-mail address:} jabueno@um.es


\begin{thebibliography}{99}


\bibitem{abro1}
U. Abresch and H. Rosenberg,
A Hopf differential for constant mean curvature surfaces in
$\mathbb S^2\times\mathbb R$ and $\mathbb H^2\times\mathbb R$,
\emph{Acta Math.} \textbf{193} (2004), 141--174.

\bibitem{abro2}
U. Abresch, H. Rosenberg,
Generalized Hopf differentials,
\emph{Mat. Contemp.} \textbf{28} (2005), 1--28.

\bibitem{ale}
A. D. Alexandrov,
Uniqueness theorems for surfaces in the large. I,
\emph{Vestnik Leningrad Univ.} \textbf{11} (1956), no. 19, 5--17
(in Russian).

\bibitem{alp}
L. J. Al\'ias, R. L\'opez and B. Palmer,
Stable constant mean curvature surfaces with circular boundary,
\emph{Proc. Amer. Math. Soc.} \textbf{127} (1999), 1195--1200.

\bibitem{bdc}
J. L. Barbosa, M. do Carmo,
Stability of hypersurfaces with constant mean curvature,
\emph{Math. Z.} \textbf{185} (1984), 339--353.

\bibitem{bdce}
J. L. Barbosa, M. do Carmo and J. Eschenburg,
Stability of hypersurfaces of constant mean curvature in Riemannian
manifolds,
\emph{Math. Z.} \textbf{197} (1988), 123--138.

%\bibitem{buloMJOM} 
%A. Bueno, R. López, 
%Stability of cylinders in $\mathbb{E}(\kappa,\tau)$ homogeneous spaces, 
%\emph{Mediterr. J. Math.} \textbf{22} (2025).

\bibitem{cohi}
R. Courant, D. Hilbert,
Methods of Mathematical Physics, vol. I,
Interscience Publishers, New York, 1953.

\bibitem{dan}
B. Daniel,
Isometric immersions into $3$-dimensional homogeneous manifolds,
\emph{Comment. Math. Helv.} \textbf{82} (2007), 87--131.

\bibitem{dhm} 
B. Daniel, L. Hauswirth, P. Mira, 
Constant mean curvature surfaces in homogeneous manifolds, 
\emph{Korea Institute for Advanced Study}, Seoul, Korea, 2009.


\bibitem{dami} 
B. Daniel, P. Mira, 
Existence and uniqueness of constant mean curvature spheres in ${\rm Sol}_3$, 
{\it J. Reine Angew. Math.}, {\bf 685} (2013), 1--32.


\bibitem{esp}
J. M. Espinar,
The spherical cap conjecture for immersed disks,
preprint (2026), arXiv:2609.24766.



\bibitem{femiTAMS} I. Fern\'andez, P. Mira, Holomorphic quadratic differentials and the Bernstein problem in Heisenberg space. 
{\it Trans. Amer. Math. Soc.}, {\bf 361}, no 11, (2009), 5737--5752.

\bibitem{femiICM} I. Fern\'andez, P. Mira, Constant mean curvature surfaces in 3-dimensional Thurston geometries. In \emph{Proceedings of the International Congress of Mathematicians}, Volume II (Invited Conferences), pages 830--861. Hindustan Book Agency, New Delhi, 2010. Preprint at http://arXiv.org/abs/1004.4752 .

\bibitem{gamiJDG} 
J.A. Gálvez, P. Mira, 
Uniqueness of immersed spheres in three-manifolds, 
\emph{J. Differential Geometry} \textbf{116} (2020), 459--480.

\bibitem{hop}
H. Hopf,
Differential Geometry in the Large,
\emph{Lecture Notes in Mathematics} \textbf{1000},
Springer-Verlag, Berlin, 1983.

\bibitem{loplibro}
R. L\'opez,
Constant Mean Curvature Surfaces with Boundary,
Springer Monographs in Mathematics,
Springer, Berlin, 2013.

\bibitem{lop1}
R. L\'opez,
Open problems on compact constant mean curvature surfaces with boundary,
in \emph{Recent Advances in Differential Geometry and Related Areas},
Contemporary Mathematics, vol. 821,
American Mathematical Society, 2025, 105--118.
arXiv:2505.01143.

\bibitem{lop2}
R. L\'opez,
Constant mean curvature disks with circular boundary in the $3$-sphere
and in hyperbolic space,
preprint (2026), arXiv:2609.39790.

\bibitem{manu}
D. Maximo, I. Nunes,
Constant mean curvature disks with circular boundary,
preprint (2026), arXiv:2609.28929.

\bibitem{mmpr}
W. H. Meeks III, P. Mira, J. Pérez and A. Ros,
Constant mean curvature spheres in homogeneous three-manifolds,
\emph{Invent. Math.} \textbf{224} (2021), 147--244.

\bibitem{sou}
R. Souam,
On stable constant mean curvature surfaces in
$\mathbb S^2\times\mathbb R$ and $\mathbb H^2\times\mathbb R$,
\emph{Trans. Amer. Math. Soc.} \textbf{362} (2010), 2845--2857.


\bibitem{tor}
F. Torralbo,
Rotationally invariant constant mean curvature surfaces in homogeneous
$3$-manifolds,
\emph{Differential Geom. Appl.} \textbf{28} (2010), no. 5, 593--607.

\bibitem{tour}
F. Torralbo, F. Urbano,
Compact stable constant mean curvature surfaces in homogeneous
$3$-manifolds,
\emph{Indiana Univ. Math. J.} \textbf{61} (2012), no. 3, 1129--1156.

\bibitem{zha}
Y. Zhang,
On the stability of constant mean curvature surfaces in
$M_k^2\times\mathbb R$ with circular boundary,
\emph{Chinese Ann. Math. Ser. A} \textbf{36} (2015), no. 4, 421--428.


\end{thebibliography}
\end{document}